\documentclass[12pt, reqno]{amsart}
\usepackage[latin1]{inputenc}
\usepackage[english]{babel}

\usepackage{mathrsfs}
\usepackage{amsthm}
\usepackage{amsmath}
\usepackage{amsfonts}
\usepackage{mathtools}
\usepackage{amssymb}

\usepackage{textcomp}
\usepackage{xcolor}

\usepackage{calrsfs}
\usepackage{csquotes}

\theoremstyle{plain}
\newtheorem {theorem}{Theorem}[section]
\newtheorem {lemma}[theorem]{Lemma}
\newtheorem {corollary} [theorem]{Corollary}

\theoremstyle{definition}
\newtheorem{definition}[theorem]{Definition}
\newtheorem{remark}[theorem]{Remark}
\newtheorem{example}[theorem]{Example}
\theoremstyle{remark}

 \numberwithin{equation}{section} 
 
\newcommand{\G}{\Gamma}
\newcommand{\R}{\mathbb{R}}

\newcommand{\N}{\mathbb{N}}
\newcommand{\Rn}{\R^n}

\newcommand{\D}{\mathcal{D}}

\newcommand{\A}{\mathcal{A}}

\newcommand{\E}{\mathscr{E}}

\newcommand{\oo}{\infty}
\newcommand{\ov}{\overline}
\newcommand{\ee}{\text{e}}

\renewcommand{\phi}{\varphi}
\newcommand{\e}{\varepsilon}

\newcommand{\0}{\vartheta}
\newcommand{\w}{\omega}

\renewcommand{\d}{\partial}
\renewcommand{\a}{\alpha}
\renewcommand{\b}{\beta}
\renewcommand{\k}{\kappa}
\newcommand{\g}{\gamma}
\newcommand{\la}{\lambda}
\newcommand{\de}{\delta}
\newcommand{\De}{\Delta}

\newcommand{\sgn}{\text{sgn}}
\renewcommand{\l}{\ell}

\newcommand{\kk}{\dot{\kappa}}
\renewcommand{\gg}{\dot{\gamma}}
\newcommand{\pphi}{\dot{\varphi}}

\newcommand{\Dev}{\mathrm{D}}
\newcommand{\cut}{\mathrm{cut}}

\newcommand{\id}{\mathrm{id}}

\renewcommand{\.}{\dots}

\begin{document}

\title[Non-minimality of spirals]{Non-minimality of spirals in sub-Riemannian manifolds}

\author[R.~Monti]{Roberto Monti} 
\email{monti@math.unipd.it}

        \address
{Universit\`a di Padova, Dipartimento di Matematica  ``Tullio Levi-Civita'',
via Trieste 63, 35121 Padova, Italy}

 \author[A.~Socionovo]{Alessandro Socionovo} 
\email{socionov@math.unipd.it}

        \address
{Universit\`a di Padova, Dipartimento di Matematica  ``Tullio Levi-Civita'',
via Trieste 63, 35121 Padova, Italy}

\begin{abstract}  We show that in analytic sub-Riemannian manifolds of rank 2 satisfying a commutativity condition spiral-like curves are not length minimizing near the center of the spiral. The proof relies upon the delicate construction of a competing curve.
\end{abstract}

\maketitle


\section{Introduction}

The regularity of geodesics (length-minimizing curves) in sub-Riemannian geometry 
is an open problem   since forty years. Its difficulty is due to the presence of singular (or abnormal) extremals, i.e.,
curves where the differential of the end-point map is singular (it is not surjective).
There exist singular curves that are as a matter of fact length-minimizing. The first example was discovered in \cite{Montgomery} and other classes of examples (regular abnormal extremals) are studied in \cite{SussmannLiu}. All such examples are smooth curves.

When the end-point map is singular, it is not possible to deduce  the  Euler-Lagrange equations with their regularizing effect for minimizers constrained on a nonsmooth set. On the other hand, in the case of singular extremals the necessary conditions given  by  Optimal Control Theory (Pontryagin Maximum Principle) do not provide in general  any further regularity beyond the starting one, absolute continuity or Lipschitz continuity of the curve.

The most elementary kind of singularity for a Lipschitz curve
is of the corner-type: at a given point, the curve has a left and a right tangent that are linearly independent. In \cite{LeonardiMonti} and \cite{LeDonnneHakavouri1} it was proved that length minimizers cannot have singular points of this kind. These results have been  improved in \cite{MontiPigatiVittone}: at any point, the tangent cone to a length-minimizing curve contains at least one line (a half line, for extreme points), see also \cite{LeDonneHakavouri2}. The uniqueness of this tangent line for length minimizers is an open problem. Indeed, there exist other types of singularities related to the non-uniqueness of the tangent. In particular,   there exist spiral-like curves whose tangent cone at the center  contains 
many and in fact all tangent lines, see Example \ref{no2} below. 
These curves may appear as Goh extremals in Carnot groups, see \cite{LDLMV1}
and \cite[Section 5]{LLMV13}.   For these reasons, 
the results  of \cite{MontiPigatiVittone} are not  enough to prove the nonminimality of spiral-like extremals. Goal of this paper is to show that curves with this kind of singularity are not length-minimizing.

Let $M$ be an $n$-dimensional, $n\geq 3$,  analytic manifold endowed with a rank 2 analytic distribution $\mathcal D \subset TM$ that is bracket generating (H\"ormander condition).
An absolutely continuous curve   $\gamma\in  AC([0,1];M)$ is horizontal if $\dot\gamma \in \mathcal D(\gamma)$ almost everywhere. The length of $\gamma$ is defined fixing a metric tensor $g$ on $\mathcal D$ and letting
\begin{equation} \label{LUNG}
  L(\gamma) =\int_{[0,1]} g_\gamma (\dot\gamma,\dot\gamma )^{1/2} dt.
\end{equation}
The curve  $\gamma$ is a length-minimizer between its end-points 
if for any other horizontal curve $\bar \gamma\in AC([0,1];M)$ such that
$\bar \gamma (0) = \gamma(0)$  and 
$\bar \gamma (1) = \gamma(1)$ we have
$L(\gamma)\leq L(\bar \gamma)$.

Our notion  of horizontal spiral in a sub-Riemannian manifold of rank 2 is fixed in Definition \ref{spiral}. 
We will show that spirals are not length-minimizing when the horizontal distribution $\mathcal D$ satisfies the  following  commutativity condition.
 Fix two  vector fields $X_1,X_{2}\in\D$ that are linearly independent at some point $p\in M$.   For $k\in\N$ and for a multi-index $J=(j_1,\dots,j_k)$, with $ j_i\in \{1,2\} $,  we denote by
 $ X_J=[X_{j_1},[ \dots,[  X_{j_{k-1} }, X_{j_k}]\cdots ]]$ the iterated commutator associated with $J$. We define its length   as $\mathrm{len}(X_J)=k$.  Let   $\D_k(p)$ be the $\R$-linear span of $\{X_J(p)\,|\,\mathrm{len}(X_J)\leq k\}\subset T_p M$. 
In a neighborhood of the center of the spiral,  we will assume  the following  condition
 \begin{equation}
\label{12ML}
 [\D_i ,\D_j ]=\{0\}\quad 
 \textrm{for all $i,j\geq2$} .
\end{equation}

Our main result is the following

 \begin{theorem} 
 \label{MAIN}
 Let $(M,\mathcal D,g)$  be an analytic sub-Riemmanian manifold of rank 2
 satisfying \eqref{12ML}.
Any horizontal  spiral $\gamma \in AC( [0,1];M)$ is not length-minimizing near its center. 
\end{theorem}
 
Differently from \cite{LeonardiMonti,LeDonnneHakavouri1,MontiPigatiVittone,LeDonneHakavouri2} and similarly to \cite{mo}, the proof of 
this theorem cannot be
reduced to the case of Carnot groups, the infinitesimal models of equiregular sub-Riemanian manifolds. This is because  the blow-up of the spiral could be a horizontal line, that is indeed length-minimizing. 

The nonminimality of spirals combined with the necessary conditions given by Pontryagin Maximum Principle is likely to give new regularity results on classes of sub-Riemannian manifolds, in the spirit of \cite{BarilariJeanPrandiSigalotti}.  We think, however, that the main interest of Theorem \ref{MAIN}
is in the deeper understanding that it provides on the loss of minimality caused  by   singularities.

The proof of Theorem \ref{MAIN} 
consists in constructing  a competing curve  shorter than the spiral.
The construction uses exponential coordinates of the second type and 
 our first step is a review of Hermes'  theorem on the structure of vector-fields in  such coordinates. In this situation, the commutativity condition \eqref{12ML}
 has a clear meaning explained in Theorem \ref{x1x2},   that may be of independent interest.
 Even though our definition of
 \enquote{horizontal spiral} is   given in coordinates of the second type, see Definition \ref{spiral},   it is actually coordinates-independent, see Remark \ref{FREE}.

In Section \ref{sez2}, we start the construction of the competing curve. Here we use the specific structure of a spiral. The gain of length is obtained by cutting one spire near the center. The adjustment of the end-point  will be obtained modifying the spiral in a certain number of locations adding \enquote{devices} depending on a set of  parameters. The horizontal coordinates of the spiral are a planar curve intersecting the positive $x_1$-axis infinitely many times. The possibility of adding devices at such locations arbitrarily close to the origin will be a crucial fact.

In Section \ref{four}, we develop an integral calculus on monomials that is used to estimate the effect of cut and devices on the end-point of the modified spiral.  Then, in Section \ref{SYS}, we fix the parameters of the devices in such a way that the end-point of the modified curve coincides with the end-point of the spiral. This is done in  Theorem \ref{stilx} by a linearization argument. Sections \ref{sez2}--\ref{SYS} contain the technical core of the paper.

We use the specific structure of the length-functional in  Section \ref{seven}, where we prove that the modified curve is shorter than the spiral, provided that the cut is sufficiently close to the origin. This will be the conclusion of the proof of Theorem \ref{MAIN}.

We briefly comment  on the assumptions made in Theorem \ref{MAIN}.
The analyticity of $M$ and $\mathcal D$ is needed only in Section \ref{due}. 
In the analytic case, it is known that length-minimizers
are smooth in an open and dense set, see \cite{Sus14}. See also \cite{FR} for a $C^1$-regularity result when $M$ is an analytic manifold of dimension $3$.

The assumption that the distribution $\mathcal D$ has rank 2 is natural when considering horizontal spirals. When the rank is higher there is room for more complicated singularities in the horizontal coordinates, raising challenging questions about the regularity problem.

Dropping the commutativity assumption \eqref{12ML} is a major technical problem: getting sharp estimates from below for the  effect produced by cut and devices on the end-point seems extremely difficult when the coefficients of the horizontal vector fields depend also on nonhorizontal coordinates, see Remark \ref{TECH}.

\section{Exponential coordinates at the center of the spiral}

\label{due}

 In this section, we introduce in $M$ exponential coordinates 
 of the second type centered at a point $p \in M$, that will be  the center of the spiral.

   Let $X_1, X_2\in \mathcal D$ be linearly independent at $p$.
 Since the distribution  $\D$ is bracket-generating  we can find vector-fields $X_3,\ldots, X_n$, with $n = \mathrm{dim}(M)$, such that
 each $X_i$ is an iterated commutator of $X_1,X_2$ with length $w_i =\mathrm{len}(X_i)$, $i=3,\ldots,n$, and such that $X_1,\ldots,X_n$ at $p$ are  a basis for $T_pM$.
By continuity, there exists an open neighborhood $U$ of $p$  such that $X_1(q),\dots,X_n(q)$ form  a basis for $T_qM$, for any $q\in U$. We call $X_1,\ldots,X_n$ a stratified basis of vector-fields in $M$.

 Let   $\phi\in C^\oo(U;\R^n)$ be a chart such that $\phi(p)=0$ and $\phi(U)=V$, with $V\subset \R^n $ open neighborhood of $0\in\R^n$.   Then
  $\widetilde X_1=\phi_*X_1,\ldots, \widetilde X_n=\phi_*X_n$  is  a system of point-wise linearly independent vector fields in   $V\subset\R^n$. Since our problem has a local nature,
we can  without loss of generality assume  that $M=V= \R^n$ and $p=0$.

After these identifications, we have a stratified basis of vector-fields  $X_1,\dots,X_n$ in $\Rn$.
We  say that  $x=(x_1,\.,x_n)\in\R^n$ are exponential coordinates of the second type
associated with the vector fields $X_1,\.,X_n$ if we have  
 \begin{equation} \label{23}
  x = \Phi_{x_1} ^{ X_1} \circ\.\circ\Phi_{ x_n}^{X_n} (0),\quad x\in \R^n.
 \end{equation}
We  are using the notation   $\Phi_s^{X}=\exp(sX)$, $s\in\R$, to denote the flow of a vector-field $X$.
 From now on, we assume that $X_1,\ldots,X_n$ are complete and induce  exponential coordinates of the second type.

 We define
the homogeneous degree of the coordinate $x_i$ of $\Rn$ as  $
    w_i=\mathrm{len}(X_i)$.
    We introduce the $1$-parameter group of dilations
    $\de_\la:\Rn\to\Rn$,  $\la>0$,
    \[
    \de_\la(x)=(\la^{w_1}x_1,\dots,\la^{w_n}x_n),\qquad x\in\R^n,
   \]
   and we say   that
    a function $f:\Rn\to\R$ is  $\de$-homogeneous   of degree $w\in\N$  if $
    f(\de_\la(x))=\la^wf(x)$ for all $x\in \R^n$ and $\la>0$.
 An example of $\de$-homogeneous function of degree $1$ is the pseudo-norm
 \begin{equation}
  \|x\| =\sum_{j=1}^n{|x_i|^{1/w_i}},\quad x\in \Rn.
 \end{equation}

The following theorem is proved in \cite{hermes} in the case of general rank.

 \begin{theorem}
 \label{giallo}
  Let $\D = \mathrm{span}\{X_1, X_2\}\subset TM $ be an analytic distribution of rank 2. 
  In exponential coordinates
  of the second type around a point $p\in M$ identified with $0\in\R^n$, the  vector fields $X_1$ and $X_2$ have the form
  \begin{equation} \label{X1X2}
   \begin{split}
    &X_1(x)=\partial_{x_1}, \\
    &X_2(x)=\partial_{x_2}+\sum_{j=3}^n{a_{j}(x)\partial_{x_j}},
   \end{split}
  \end{equation}
 for $x\in U$, where $U$ is a neighborhood of $0$. The analytic  functions    $a_{j}\in C^\oo(U)$, $j=3,\ldots,n$, have the structure    $a_{j}=p_{j}+r_{j}$, where:
  \begin{itemize}
   \item[(i)] $p_{j}$ are $\de$-homogeneous polynomials  
   of degree $w_j-1$
     such that $p_j(0,x_2,\dots,x_n)=0$; 
         \item[(ii)] $r_{j}\in C^\oo(U)$ are analytic functions such that,
   for some constants $C_1,C_2>0$ and for $x\in U$,
   \begin{equation}
   \label{C1C2}
     |r_{j}(x)|\leq C_1\|x\|^{w_j} \quad \textrm{and}\quad  |\d_{x_i } r_{j}(x)|\leq C_2\|x\|^{w_j-w_i}.
     \end{equation}
      \end{itemize}
 \end{theorem}

\begin{proof}  
The proof that $a_{j}=p_{j}+r_{j}$ where $p_{j}$ are polynomials  as in (i)  
and the remainders $r_{j}$ are real-analytic functions such that  $r_{j}(0)=0$ can be found in \cite{hermes}.
The proof of (ii) is also implicitly contained in \cite{hermes}. Here, we add some details.
The   Taylor series of $r_j$ has the  form
 \[
  r_j(x)=\sum_{\l =w_j}^\infty r_{j\l}(x)=
  \sum_{\l= w_j}^\infty \sum_{\a\in\A_\l}c_{\a \l}  x^\a,
 \]
 where $\A_\l = \{\alpha\in \N^n:  \alpha_1 w_1 + \ldots +\alpha_n w_n =\l\}$, 
 $x^\a = x_1^{\alpha_1}\cdots x_n^{\alpha_n}$ and 
  $c_{\a\l}\in\R$ are constants. Here and in the following, $\N =\{0,1,2,\ldots\}$.  The series converges absolutely in a small homogeneous cube $Q_\delta =\{ x\in \Rn: \| x\| \leq \delta\}$ for some $\delta>0$, and in particular
 \[
  \sum_{\l= w_j}^\infty \delta ^\l\sum_{\a\in\A_\l}|c_{\a\l} |<\oo.
 \]
 Using the inequality  $|x^\a|\leq\|x\|^\l$ for $\a\in\A_\l$,  for $x\in Q_\delta$ we get 
  \[
  |r_j(x)|\leq C_1 \|x\|^{w_j}, \quad \textrm{with }C_1 =  \sum_{\l=w_j}^\infty \delta ^{\l-w_j}\sum_{\a\in\A_\l}|c_\a|<\infty .
 \]
 
  The estimate for the derivatives of $r_j$ is analogous. Indeed, we have 
   \[
  \d_{x_i}r_j(x)=\sum_{\l =w_j} ^\infty 
  \sum_{\a\in\A_\l}  \alpha_i c_{\a\l} x^{\a-\ee_i},
 \]
 where $\a-\ee_i \in \A_{\l-w_i} $ whenever $\a\in \A_\l$. Thus the leading term in the series has homogeneous degree $w_j-w_i$ and repeating the argument above we get the estimate $|\d_{x_i } r_{j}(x)|\leq C_2\|x\|^{w_j-w_i}$ for $x\in Q_\delta$.

 \end{proof}

When the distribution $\mathcal D$ satisfies the commutativity assumption \eqref{12ML}
the coefficients $a_j$ appearing in the vector-field $X_2$ in \eqref{X1X2}
enjoy additional properties.

\begin{theorem}
\label{x1x2}
 If $\mathcal D\subset T M$ is an analytic distribution of rank 2 satisfying \eqref{12ML} then the functions $a_3,\ldots,a_n$ of Theorem \ref{giallo} depend only on the variables $x_1$ and $x_2$.
\end{theorem}

\begin{proof}

Let $\Gamma:\R\times\R^n\to\Rn$ be the map $ \Gamma(t,x)
   =\Phi_{t}^{X_2}(x), 
$
 where $ x\in\Rn$ and $   t\in\R$. Here, we   are using the exponential coordinates \eqref{23}.  
 In the following we omit the composition sign $\circ$. 
Defining 
 $\Theta:\R^3\times\R^n\to\R^n$ as the map   
 $
 \Theta_{t,x_1,x_2}(p)=\Phi_{-(x_2+t)}^{X_2}\Phi_{-x_1}^{X_1}\Phi_{t}^{X_2}\Phi_{x_1}^{X_1}\Phi_{x_2}^{X_2}(p),
 $ 
 we  have  
\[
 \G(t,x)=\Phi_{x_1}^{X_1}\Phi_{x_2+t}^{X_2}\Theta_{t,x_1,x_2}\Phi_{x_3}^{X_3}\.\Phi_{x_n}^{X_n}(0).
\]
We claim that there exists a   $C>0$ independent of $t$ such that, for $t\to0$,
  \begin{equation}
\label{commuta}
|   \Theta_{t,x_1,x_2}\Phi_{s}^{X_j}-\Phi_{s}^{X_j}\Theta_{t,x_1,x_2}| \leq  C t^2.
 \end{equation} 
We will prove   claim \eqref{commuta} in Lemma \ref{leo} below. 
From \eqref{commuta} it follows that  there exist  mappings  $R_{t}\in C^\infty(\Rn,\Rn)$ such that 
\begin{equation} \label{commuta2}
 \G(t,x)=\Phi_{x_1}^{X_1}\Phi_{x_2+t}^{X_2}\Phi_{x_3}^{X_3}\.\Phi_{x_n}^{X_n}\Theta_{t,x_1,x_2}(0)+R_t(x),
\end{equation}
and such that $|R_t|\leq C t^2$ for $t\to0$.

By the structure 
\eqref{X1X2} of the vector fields $X_1$ and $X_2$
and since
$\Theta_{t,x_1,x_2}$ is the composition of $C^\oo$ maps,  
there exist $C^\oo$ functions $f_j=f_j(t,x_1,x_2)$ such that
\begin{equation}
\label{TETA}
 \Theta_{t,x_1,x_2}(0)=\big(
 0,0,f_3(t,x_1,x_2),\.,f_n(t,x_1,x_2)\big)
 =\exp\Big(\sum_{j=3}^nf_j(t,x_1x_2)X_j\Big)(0).
\end{equation}
By \eqref{12ML}, from  \eqref{commuta2} and \eqref{TETA} we obtain 
\begin{align*}
 \Gamma(t,x)&=  \Phi_{x_1}^{X_1} \Phi_{x_2+t}^{X_2} \exp\Big(\sum_{i=3}^n(x_j+f_j(t,x_1,x_2))X_j\Big)(0) +R_t(x)  
 \\
 &=\big(
 x_1,x_2+t,x_3+f_3(t,x_1,x_2),\.,x_n+f_n(t,x_1,x_2)\big) +R_t(x) ,
\end{align*}
 and we conclude that     
 \[
  X_2(x) =
  \frac{d}{dt}\Gamma(x,t)\Big|_{t=0}=\d_2+\sum_{j=3}^n\frac{d}{dt}f_j(t,x_1,x_2)\Big|_{t=0}\d_j.
 \]
 Thus  the coefficients  $a_j(x_1,x_2)=\frac{d}{dt}f_j (t,x_1,x_2)|_{t=0}$, $j=3,\ldots,n$, depend only on the first two variables, completing the proof.
\end{proof}

In the  following lemma, we   prove our claim \eqref{commuta}.

\begin{lemma}
\label{leo}
 Let $\mathcal D \subset T M$ be an analytic distribution satisfying \eqref{12ML}.
 Then for  any $j=3,\ldots,n$  the claim in \eqref{commuta} holds.
\end{lemma}

\begin{proof}  Let $X=X_j$ for any $j=3,\ldots,n$ and    
  define the map  $T_{t,x_1,x_2;s} ^X =\Theta_{t,x_1,x_2}\Phi_s^X-\Phi_s^X\Theta_{t,x_1,x_2}$. For $t=0$ the map $\Theta_{0,x_1,x_2}$ is the identity and thus $T_{0,x_1,x_2;s} ^X=0$. So, claim \eqref{commuta}
 follows as soon as we show that  
 \[
  \dot T_{0,x_1,x_2;s} ^X= \frac{\d}{\d t}\Big|_{t=0}T_{t,x_1,x_2;s} ^{X} =0,
 \]
 for any $s\in\R$ and for all $x_1,x_2\in\R$.

 We first 
  compute the derivative of $\Theta_{t,x_1,x_2}$ with respect to $t$.
 Letting $\Psi_{t,x_1}=\Phi_{-x_1}^{X_1}\Phi_{t}^{X_2}\Phi_{x_1}^{X_1}$ we have  
 $
  \Theta_{t,x_1,x_2}=\Phi_{-(x_2+t)}^{X_2}\Psi_{t,x_1}\Phi_{x_2}^{X_2},
 $
 and, thanks to \cite[Appendix A]{hermes}, the derivative of $\Psi_{t,x_1}$ at $t=0$ is 
 \[
  \dot\Psi_{0,x_1}=\sum_{\nu=0}^\oo c_{\nu,x_1} W_\nu ,  \]
 where $W_\nu = [X_1,[ \cdots,[X_1,X_2] \cdots]]$ with $X_1$ appearing $\nu$ times and $c_{\nu,x_1}=(-1)^\nu x_1^\nu/\nu!$.
 In particular, we have $c_{0,x_1}=1$. Then the derivative of $\Theta_{t,x_1,x_2}$ at $t=0$ is
 \begin{align*}
  \dot\Theta_{0,x_1,x_2}&=-X_2+ d \Phi_{-x_2}^{X_2} \big(  \dot\Psi_{0,x_1} (\Phi_{x_2}^{X_2})\big)
  \\
  &
  =-X_2+\sum_{\nu=0}^\oo c_{\nu,x_1} d \Phi_{-x_2}^{X_2}\big(
  W_\nu ( \Phi_{x_2}^{X_2})\big)
  \\
  &=\sum_{\nu=1}^\oo c_{\nu,x_1} d \Phi_{-x_2}^{X_2}\big(
  W_\nu ( \Phi_{x_2}^{X_2})\big),
 \end{align*}
 because   the term in the sum with $\nu=0$ is  $d \Phi_{-x_2} ^{X_2} \big(X_2(\Phi_{x_2}^{X_2}) \big)= X_2$.
  Inserting this formula for $
  \dot\Theta_{0,x_1,x_2}$ into  
 \begin{equation}
 \label{titti}
  \dot T_{0,x_1,x_2;s} ^X =\dot\Theta_{0,x_1,x_2} (\Phi_s^X)- d \Phi_s^X( \dot \Theta_{0,x_1,x_2}),
 \end{equation} 
   we obtain
 \begin{align*}
  \dot T_{0,x_1,x_2; s} ^{X}=&\sum_{\nu=1}^\oo c_{\nu,x_1}
  d \Phi_{-x_2}^{X_2} \big ( W_\nu (\Phi_{x_2}^{X_2}\Phi_s^X ) \big)
  -d \Phi_s^{X}  \sum_{\nu=1}^\oo c_{\nu,x_1}
  d \Phi_{-x_2}^{X_2}\big( W_\nu\big(\Phi_{x_2}^{X_2})\big)
  \\
  =& d \Phi_s^X 
  \sum_{\nu=1}^\oo c_{\nu,x_1}\Big(
  d \Phi_{-s}^X  d \Phi_{-x_2}^{X_2}
  \big(    W_\nu  (\Phi_{x_2}^{X_2}\Phi_s^X) \big)
  -  d \Phi_{-x_2}^{X_2} \big(
   W_\nu(\Phi_{x_2}^{X_2})\big)\Big).
 \end{align*}
In order  to prove that 
$  \dot T_{0,x_1,x_2; s} ^{X}$ vanishes for all $x_1,x_2$ and $s$, we have to show that
\begin{equation}
\label{CLAM}
  g(x_2,s):=  d \Phi_{-s}^X  d \Phi_{-x_2}^{X_2}
  \big(    W_\nu  (\Phi_{x_2}^{X_2}\Phi_s^X) \big)
  -  d \Phi_{-x_2}^{X_2} \big(
   W_\nu(\Phi_{x_2}^{X_2})\big)=0,
\end{equation}
for any $\nu\geq1$ and for any $x_2$ and $s$.
From 
$\Phi_0^{X}=\id$
it follows that  $ g(x_2,0)=0$.
Then, our claim \eqref{CLAM} is implied by  
\begin{equation}
 \label{dsh}
h(x_2,s) :=  \frac{\d}{\d s} g(x_2,s)=0.
\end{equation}
Actually, 
this is a Lie derivative and, namely,
 \begin{align*}
   h(x_2,s)&
  =- d \Phi_{-s}^X  \big[X, d \Phi_{-x_2}^{X_2} \big(  W_\nu  (\Phi_{x_2}^{X_2})\big)\big].
\end{align*}
Notice that $h(0,s) =
- d \Phi_{-s}^X  [X, W_\nu ] =0$ by our assumption \eqref{12ML}. In a similar way, 
for any $k\in\N$ we have 
\[
\frac{\d ^ k }{\d x_2^k } h(0,s) =
(-1)^ {k +1} d \Phi_{-s}^X  [X, [X_2,\cdots [X_2, W_\nu] \cdots ] ] =0,
\]
with $X_2$ appearing $k$ times.  Since the function $x_2\mapsto h(x_2,s)$ is analytic
our claim \eqref{dsh} follows.

 \end{proof}

From now on, we assume that $a_ j (x) = a_j(x_1,x_2)$ are functions of the variables $x_1,x_2$.

A curve $\gamma\in AC([0,1];M)$  is horizontal if $\dot \gamma(t) \in \D(\gamma(t) )$ for a.e.~$t\in [0,1]$. In exponential coordinates we have $\gamma = (\gamma_1,\ldots,\gamma_n)$ where, for $j=3,\ldots, n$, the coordinates satisfy the following integral identities  
 \begin{equation}
\label{HL}   \gamma_j(t) = \gamma_j(0) + \int_0^t a_j(\gamma_1(s), \gamma_2(s) ) \dot\gamma_2(s) ds,\quad t\in [0,1].
 \end{equation}
When $\gamma(0)$  and $\gamma_1,\gamma_2$ are given, these formulas  determine in a unique way the whole horizontal curve $\gamma$. 
 We call $\kappa\in AC([0,1];\R^2)$, $\kappa =(\gamma_1,\gamma_2)$, the horizontal coordinates of $\gamma$.   
 
\begin{definition}[Spiral] 
\label{spiral} 
 We say that a horizontal curve $\gamma\in AC([0,1];M)$ is a
\emph{spiral}
  if, in exponential coordinates of the second type centered at $\gamma(0)$, the horizontal
coordinates $\kappa\in AC([0,1];\R^2)$ are of the form
 \begin{equation}
  \label{eq2}
  \kappa(t)=t\mathrm {e}^{i\phi(t)},
   \quad t\in]0,1],
 \end{equation}
where  $\phi\in C^1(  ]0,1]; \R^+)$ is a function, called  \emph{phase}
  of the spiral, such that $|\phi(t)|\to\oo$ and 
  $|\dot\phi(t)| \to\infty$ 
  as $t\to0^+$.
  \end{definition} 
  
  Without loss of generality, we shall focus our attention 
  on spirals that are oriented clock-wise, i.e., with a phase satisfying 
   $\phi(t)\to\oo$ and 
  $\dot\phi(t) \to-\infty$ 
  as $t\to0^+$.  Such a phase is decreasing 
near $0$.  Notice that if $\phi(t)\to \infty$ and $\dot  
  \phi(t)$ has a limit as $t\to0^+$ then this limit must be $-\infty$.

 \begin{example}
  \label{no2}
  An  interesting example of horizontal spiral is the double-logarithm spiral, the horizontal lift
  of the curve $\kappa$ in the plane of the form \eqref{eq2} with   phase  $\phi(t)=\log(-\log t)$, $t\in (0,1/2]$.  In this case, we have
  \[
  \dot\phi(t) = \frac{1}{t \log t }, \quad t\in (0,1/2],
  \]
and clearly $\phi(t)\to\infty$ and $\dot\phi(t)\to-\infty$ as $t\to0^+$.
In fact, we also have $t\dot\phi \in L^\infty(0,1/2)$, which means that $\kappa$ and thus $\gamma$ is Lipschitz continuous.
This spiral has the following additional properties:
  \begin{itemize}
   \item[i)]  for any $v\in \R^2$ with $|v|=1$
   there exists an  infinitesimal sequence of positive real numbers $(\lambda_n)_{
   n\in\N}$ such that $\kappa(\lambda_n t)/\lambda _n \to tv $ locally uniformly, as $n\to\infty$;

    \item[ii)] for any    infinitesimal sequence of positive real numbers $(\lambda_n)_{
   n\in\N} $ there exists a subsequence and a
   $v\in \R^2$ with $|v|=1$ such that $\kappa(\lambda_{n_k} t)/\lambda _{n_k} \to tv $ as $k\to\infty$, locally uniformly.
   \end{itemize}
   This means that the  tangent cone of $\kappa$ at $t=0$ consists of all half-lines in $\R^2$ emanating from $0$. 
    
 \end{example}

 \begin{remark} \label{FREE}
 We show that  Definition \ref{spiral}
 of a horizontal spiral does not  in fact depend on the chosen coordinates.
 
 Let  $F\in C^\infty(\R^n;\Rn)$ be a diffeomorphism such that $F(0)=0$ and $d_0 F(  \R^2\times\{0\}) =\R^2\times\{0\} $, where $d_0 F$ is the differential of $F$ at $0$. In the new coordinates, the spiral $\gamma$ becomes  $\zeta (t) = F(\gamma(t))$ with horizontal coordinates $\xi (t) = (F_1(\gamma(t)), F_2(\gamma(t)))$. We claim that after a reparameterization $\xi$ is of the form \eqref{eq2}, with a phase $\omega$ 
 satisfying $|\omega |\to \infty$ and $ |\dot{\omega}|\to\infty$. In particular, we will show that $|\dot\omega |\to\infty$.
 
 The function $s(t)  = |\xi(t) | = |(F_1(\gamma(t)), F_2(\gamma(t)))|$
 satisfies
 \begin{equation}
 \label{pippo}
      0<c_0 \leq \dot s(t) \leq c_1<\infty,\quad t\in (0,1].
 \end{equation}
 Define the function $\w \in  C^1((0,1])$ letting  $\xi(t) = s(t) \mathrm{e} ^{i \w(s(t))}$. Then differentiating the identity obtained inverting    
 \[
\tan \!  \big(   \omega(s(t)) \big) =  \frac{F_2(\gamma(t)) }{F_1(\gamma(t)) },\quad t\in (0,1],
  \]
  we obtain 
 \begin{equation}
 \label{topolino}
  \dot s(t) \dot{\w}(s(t)) = 
  \frac{1}
  { s (t)^2}
\langle  \Phi(\gamma(t)) 
,\dot\gamma (t)\rangle,\qquad t\in(0,1],
 \end{equation}
where the function $\Phi(x)  = 
F_1 (x)  \nabla F_2(x) 
 -F_2 (x)  \nabla F_1( x) 
$ has the Taylor development as $x\to 0$
\[
\begin{split}
\Phi(x)    
 = &  \langle \nabla F_1(0),x\rangle   \nabla F_2(0) 
 -  \langle \nabla F_2(0),x\rangle    \nabla F_1(0) + O(|x|^2) .
\end{split}
\]
 Observe that from \eqref{HL} it follows that  $|\dot\gamma_j(t)| = O(t)$ for $j\geq 3$.
 Denoting by $\bar\nabla$ the gradient in the first two variables, we deduce that as $t\to0^+$ we have 
 \begin{equation}\label{zx}
\langle  \Phi(\gamma) 
,\dot\gamma \rangle = 
\langle  F_1 (\gamma) \bar  \nabla F_2( \gamma) 
 -F_2 (\gamma)  \bar \nabla F_1( \gamma)
,\dot\kappa \rangle + O(t^2)
 \end{equation}
 with
 \[
 F_1 (\gamma) \bar  \nabla F_2( \gamma) 
 -F_2 (\gamma)  \bar \nabla F_1( \gamma)
=  \langle \bar  \nabla F_1(0),\kappa \rangle  \bar  \nabla F_2(0) 
 -  \langle \bar \nabla F_2(0),\kappa  \rangle \bar   \nabla F_1(0)
+O(t^2).
 \]
 Inserting the last identity and  $\dot\kappa = \mathrm{e} ^{i\phi} + i t \dot\phi 
 \mathrm{e} ^{i\phi}$
 into \eqref{zx}, after some computations we obtain 
 \[
\langle \Phi(\gamma) ,\dot\gamma\rangle =  \dot\phi t^2  \det (d_0 \bar F   (0)) + O(t^2),
 \]
 where $ \det (d_0 \bar F   (0))\neq 0$ is the determinant 
 Jacobian at $x_1=x_2=0$ of the  mapping $(x_1,x_2)\mapsto (F_1(x_1,x_2,0), F_2(x_1,x_2,0))$.
 Now the claim $|\dot{\w}(s)|\to\infty$ as $s\to0^+$ easily follows
 from \eqref{pippo}, \eqref{topolino} and from $|\dot\phi(t)|\to \oo$ as $t\to0^+$.

 \end{remark}

\section{Cut and correction devices} \label{sez2}

In this section, we begin the construction of the competing curve.
Let $\gamma$ be a  spiral with horizontal coordinates $\kappa$ as in \eqref{eq2}.
We can assume that $\phi$ is decreasing and that   $\phi(1)=1$ and we denote by
  $\psi:[1,\oo)\to(0,1]$ the inverse function of  $\phi$.  
For $k\in \N$ and $\eta \in[0,2\pi)$ we define $t_{k \eta}\in (0,1]$ as the unique solution to the equation $\phi(t_{k\eta  }) = 2\pi k + \eta$, i.e., we let $t_{k\eta } = \psi (2\pi k + \eta)$. 
The times
\begin{equation}
\label{tk}
t_k = t_{k0}= \psi(2\pi k) ,\quad k\in\N,
\end{equation}
will play a special role in our construction.  
The points $\kappa(t_k)$ are in the positive $x_1$-axis.

For a fixed $k\in\N$, we cut the curve $\k$ in the interval $[t_{k+1},t_{k}]$  following the line segment joining   $\k(t_{k+1})$ to $\k(t_{k})$ instead of the path $\k$, while we leave unchanged the remaining part of the path. We call this new curve $\k_{k}^{\cut}$ and, namely, we let 
 \[
 \begin{split}
  \k_{k}^{\cut
  }(t) & =\kappa(t)\quad \text{for}\quad  t\in [0,t_{k+1}]\cup [t_k, 1],
  \\
  \k_{k}^{\cut
  }(t) & =(t,0)
  \quad \text{for}\quad   t\in [t_{k+1},t_k]. 
 \end{split}
 \]
We denote by   $\g _{k}^{\cut}\in AC ([0,1];M)  $ the  horizontal curve with horizontal coordinates $\k_{k}^{\cut}$ and such that   $\g_{k}^{\cut}(0)=\g(0)$. For $t\in[0,t_{k+1}]$, we have $\g_{k}^{\cut}(t)=\g(t)$. 
To correct the errors produced by the cut on the end-point, we modify the curve $\k_{k}^{\cut}$ using a certain number of devices. The construction is made by induction.

We start with the base construction.  
Let $\E  = (h,\eta,\e)$ be a triple such that $h\in \N$,  $0<\eta<\pi/4$, and $\e \in \R $. Starting from a curve $\kappa:[0,1]\to \R^2$, we define the curve $\Dev(\kappa;\E):[0,1+2|\e|]\to\R^2$ in the following way:
 \begin{equation}
 \label{duello}
   \Dev(\kappa;\E) (t) =
   \left\{
   \begin{array}{ll}
   \kappa(t) & t\in [0,t_{h\eta}]
   \\
   \kappa(t_{h\eta}) + (\sgn(\e)(t-t_{h\eta}),0) & t\in [t_{h\eta},t_{h\eta}+|\e|]
   \\ 
   \kappa(t-|\e|) + (\e,0) & 
   t\in [t_{h\eta}+|\e|, t_{h}+|\e|]
   \\ 
   \kappa( t_{h}) + (2\e+\sgn(\e)(t_{h}-t), 0 )& 
   t\in [t_{h}+|\e|,t_{h}+2|\e|]
   \\ 
   \kappa(t- 2|\e|)& 
   t\in [t_{h}+2|\e|,1+2|\e|].
   \end{array}
   \right.
 \end{equation}
We denote by $\Dev(\g;\E)$ the horizontal curve with horizontal coordinates $\Dev(\k;\E)$. We let $\dot\Dev(\g;\E)=\frac{d}{dt}\Dev(\g;\E)$ and we indicate by $\Dev_i(\g;\E)$ the i-th coordinate of the corrected curve in exponential coordinates.
 
In the lifting formula \eqref{HL}, the intervals where $\gg_2 = 0$ do not contribute to the integral. For this reason, in \eqref{duello} we may cancel the second and fourth lines, where $\dot\Dev_2(\g;\E)=0$, and then reparameterize the curve on $[0,1]$. Namely, we define the discontinuous curve $\ov\Dev(\k;\E):[0,1]\to\R^2$ as
  \begin{equation}
   \label{cip}
   \ov\Dev(\k;\E)(t)=\left\{
   \begin{array}{ll}
    \kappa(t) & t\in [0,t_{h\eta}]
    \\ 
    \kappa(t) + (\e,0) & 
    t\in [t_{h\eta}, t_{h}]
    \\ 
    \kappa(t)& 
    t\in [t_{h},1],
   \end{array}
   \right.
  \end{equation}
and then we consider  the ``formal'' i-th coordinate  
\[
 \ov\Dev_i(\g;\E)(t)=\int_0^t a_i(\ov\Dev(\kappa;\E)(s)) \kk_2(s) ds.
\]
The following identities can be checked by  an elementary computation  (for $\e>0$)
 \begin{equation}\label{ciop}
  \ov\Dev(\g;\E)(t)=
  \left\{
  \begin{array}{ll}
   \Dev(\g;\E)(t) & t\in[0,t_{h\eta}]\\
   \Dev(\g;\E)(t+\e) & t\in[t_{h\eta},t_h]\\
   \Dev(\g;\E)(t+2\e) & t\in[t_h,1].
  \end{array}
  \right.
 \end{equation}
 With this notation,  the final  error produced on the i-th coordinate by the correction device $\E$ is
\begin{equation}
\label{pluto} 
 \g_i(1)-\Dev_i(\g;\E)(1+2|\e |)=\int_0^1 \big\{a_i(\kappa(s))-a_i(\ov\Dev(\kappa;\E)(s))\big\}\kk_2(s)
 ds.
\end{equation}
The proof of this formula is elementary and can be omitted.

We will iterate the above construction a certain number of times depending on a collections of triples $\E$. We first fix the number of triples and iterations.

For $i=3,\.,n$, let $\mathcal B_i=\{(\a,\b)\in\N^2\, :\, \a+\b=w_i-2\}$,
where $w_i\geq 2 $ is the homogeneous degree of the coordinate $x_i$.
Then,  the polynomials $p_i$ given by Theorem \ref{giallo}  and Theorem \ref{x1x2}
are of the form
 \begin{equation}
 \label{pi}
  p_i(x_1,x_2)=\sum_{(\a,\b)\in\mathcal B_i} c_{\a\b} \, x_1^{\a+1}x_2^\b,
 \end{equation}
 for suitable constants $c_{\a\b}\in\R$. 
We set    
\begin{equation}
    \l=\sum_{i=3}^n\mathrm{Card} (\mathcal B_i),
\end{equation}
and we consider an $(\ell-2)$-tuple of  triples $\bar {\E }= (\E_3,\ldots,\E_\l)$
such that $h_\l<h_{\ell-1}<\ldots <h_3<k$. Each triple is used to correct one monomial.

Without loss of generality, we   simplify the construction in the following way.
In the sum \eqref{pi}, we can assume that $c_{\a\b}=0$ for all $(\a,\b)\in \mathcal B _i$ but one.
Namely, we can assume that
\begin{equation}
\label{pipi}
     p_i(x_1,x_2) = x_1^{\a_i+1}x_2^{\b_i}\quad \textrm{with}\quad \a_i+\b_i=w_i-2,
\end{equation}
and with   $c_{\a_i\b_i}=1$. In this case, we have $\l = n$ and we will use $n-2$ devices associated with the triples $\E_3,\ldots, \E_{n}$ to correct the coordinates $i=3,\ldots,n$.
By the bracket generating property  of the vector fields $X_1$ and $X_2$ 
 and by the stratified basis property for $X_1,\ldots,X_n$, the pairs $( \a_i, \b_i)$
satisfy the following condition
\begin{equation} 
\label{INJ}
 ( \a_i, \b_i)\neq ( \a_j, \b_j)
\quad
\textrm{for}
\quad 
i\neq j.
\end{equation}
 
\noindent
From now on in the rest of the paper we will assume that the polynomials $p_i$ are of the form \eqref{pipi} with \eqref{INJ}.

Now we clarify the inductive step of our construction.
 Let $\E _3 = (h_3,\eta_3,\e_3)$ be a triple such that $ h_3< k $. We define the curve 
 $\kappa^{(3)}
 = \Dev(
  \k_{k}^{\cut};\E_3)$. Given a triple
 $\E _4 = (h_4,\eta_4,\e_4)$ with $h_4< h_3 $ we then define
 $\kappa^{(4)} 
 = \Dev(  \k ^{(3)};\E_4)$.
  By induction on $\ell\in\N$,
given a triple
 $\E _\ell = (h_\ell,\eta_\ell,\e_\ell )$ with $h_\ell < h_
 {\ell-1}$,  we   define
 $\kappa^{(\ell)} 
 = \Dev(  \k ^{(\ell-1)};\E_\ell)$. When $\ell=n$ we stop.

 We define the planar curve $\Dev(\k;k, {\bar \E }) \in AC ([0,1+2 \bar \e];\R^2)$ as $\Dev(\k;k,{\bar \E })  = \kappa^{(n)}$ according to the inductive construction explained above, where $\bar \e = |\e_3|+\ldots+ |\e_n |$.
Then we call $\Dev (  \gamma;k,\bar \E)   \in AC ([0,1+2 \bar \e];M)$,
the horizontal lift of $\Dev(\k;k,\bar \E)$ with $\Dev(\gamma;k,\E)(0) =\gamma(0)$, the modified curve of $\gamma$ associated with $\bar \E$ and with cut of parameter $k\in \N$.
There is a last adjustment to do.
In $[0,1+2 \bar \e]$
there are $2(n-2)$ subintervals where $\dot\kappa_2^{(n)}=0$. On each of  these intervals the coordinates $\Dev _j (  \gamma;k,\bar \E) 
$ are constant. According to the procedure explained in \eqref{duello}--\eqref{ciop}, we erase these intervals and we parametrize the resulting curve on $[0,1]$.
We denote this curve by   $\bar \gamma = \ov  \Dev( \gamma;k,\bar \E) $.

\begin{definition} [Adjusted modification of $\g$]
We call the curve $\bar \gamma = \ov  \Dev( \gamma;k,\bar \E): [0,1]\to M$
the adjusted modification of $\gamma$ relative to the collections of devices $\bar\E=(\E_3,\ldots,\E_n)$
and with   cut of parameter  $k$.
\end{definition}

 Our next task is to compute the error produced by cut and devices on the end-point of the spiral.
For $i=3,\ldots, n$ and for $t\in [0,1]$ we let 
\begin{equation}
 \label{Delta}
 \De_i^\g(t) = a_i(\k(t))\dot\k_2(t) - a_i(\bar \k (t)) \dot{\bar \k} _2(t).
\end{equation}
When $t<t_{k+1}$ or $t>t_k$ we have $\kk_2=\dot{\bar\k}_2$ and so the definition above reads
\[
 \Delta_i^\g (t) = \big(a_i(\k(t)) - a_i(\bar\k(t))\big)\kk_2(t).
\]

By the recursive application of the argument used to obtain \eqref{pluto}, we get the following formula for  the error at the final time $\bar t =t_{h_n}$:
\begin{equation}
\label{erri}
\begin{split}
 E_i^{k,\bar\E } & = \g_i(\bar t )-\bar \g_i(\bar t )
 =\int_{t_{k+1}}^{\bar t } \Delta_i^\g(t) dt
 \\
 &= \int_{F_k} \Delta_i^\g(t) dt + \sum_{j=3}^n
 \Big(
 \int_{A_j} \Delta_i^\g(t) dt
 +
 \int_{B_j} \Delta_i^\g(t) dt \Big).
  \end{split}
\end{equation}
In \eqref{erri} and in the following, we use the following notation for the intervals:
\begin{equation}
 F_k=[t_{k+1},t_k], \quad A_j=[t_{h_{j-1}},t_{h_j\eta_j}], \quad B_j=[t_{h_j\eta_j},t_{h_j}],
\end{equation}
with    $t_{h_2}= t_k$.
 We used also the fact that on $[0,t_{k+1}]$ we have $\gamma=\bar \gamma$.

On the interval $F_k$ we have  $\dot{\bar \k}_2 =0$ and thus
\begin{equation} \label{erri2}
\int_{F_k} \Delta_i^\g\, dt = \int_{F_k}
 \big\{  p_i(\k) +r_i(\k) \big\} \dot\k_2dt.
\end{equation} 
On the intervals $A_j$ we have $\kappa=\bar\kappa$ and thus
\begin{equation} \label{erri3}
 \int_{A_j  }  \Delta_i^\g dt = 0,
\end{equation}
because the functions $a_i$ depend only on $\k$.
Finally, on the intervals $B_j$ we have $\bar \k_1 = \k_1+ \e_j$ and $\kappa_2 =\bar\kappa_2$ and thus  
\begin{equation} \label{erri4}
 \int_{B_j}
 \Delta_i ^\g\, dt = \int_{B_j} \{p_i(\k) - p_i(\k+(\e_j,0))\} \kk_2dt +
 \int_{B_j} \{r_i(\k) - r_i(\k+(\e_j,0))\} \kk_2dt.
\end{equation}

  Our goal is to find $k\in \N$ and devices $\bar\E$ such that
$ E_i^{k,\bar\E }=0$ for all $i=3,\ldots, n$ and such that the modified curve $\Dev(\g;k,\bar\E)$ is shorter than $\g$.

\section{Effect of  cut and devices on  monomials and remainders}

\label{four}

Let  $\gamma$ be a horizontal  spiral   with horizontal coordinates $\k\in AC([0,1];\R^2)$ of the form \eqref{eq2}.   We prove some   estimates about the integrals of the polynomials \eqref{pipi} along the curve $\k$. These estimates are preliminary to the study of the errors introduced in \eqref{erri}.

For $\a,\b\in\N$, we associate with the monomial $p_{\a\b}(x_1,x_2)=
x_1^{\a+1}x_2^{\b}$ the function $   \g_{\a\b}$ defined for $t\in [0,1]$ by
 \begin{equation*}
  \begin{split}
   \g_{\a\b}(t)&=\int_{\k|_{[0,t]}}{p_{\a\b}(x_1,x_2)dx_2}=\int_0^t{p_{\a\b}(\k(s)) \dot \k_2(s)ds}.
  \end{split}
 \end{equation*}
When $p _ i = p_{\a\b}$, the function $\g_{\a\b}$ is  the leading term in the i-th coordinate of $\g$ in exponential coordinates. In this case,  the problem of estimating  $\g_i(t)$ reduces   to the estimate of integrals of the form
 \begin{equation} \label{I}
  I_{\w\eta}^{\a\b}=\int_{t_{\eta}}^{t_{\w}}{\k_1(t)^{\a+1}\k_2(t)^\b \dot \k _2(t)dt},
 \end{equation}
 where $\w \leq \eta $ are angles,  $t_\w=\psi (\w)$ and $t_\eta =\psi (\eta )$. 
 These integrals are related to the integrals
  \begin{equation}\label{J}
   J_{\w\eta}^{\a\b}=\int_{\w}^{\eta}{t_\0^{\a+\b+2}\cos^\a(\0)\sin^\b(\0)d\0}.
  \end{equation}
In the following, we will  use the short notation
 $    D^{\a\b}_\w = \cos^{\a+1}(\w)\sin^{\b+1}(\w)$.

 \begin{lemma}
 \label{L5}
 For any $\a,\b\in\N$ and  $\w\leq \eta$   we have the identity
  \begin{equation}
  \label{I->J}
   \begin{split}
    (\a+\b+2)I_{\w\eta}^{\a\b}=&t_\w^{\a+\b+2} D^{\a\b}_\w
    -t_{\eta}^{\a+\b+2} D^{\a\b}_{\eta}
   -(\a+1)J_{\w\eta}^{\a\b}.
   \end{split}
  \end{equation}
 \end{lemma}

 \begin{proof} Inserting into
 $I_{\w\eta}^{\a\b}$
   the identities $\k_1(t)=t\cos(\phi(t))$, $\k_2(t)=t\sin(\phi(t))$, and $\dot\k  _2(t)=\sin(\phi(t))+t\cos(\phi(t))\pphi(t)$ we get
  \[
   I_{\w\eta}^{\a\b}=\int_{t_{\eta }}^{t_{\w}}{t^{\a+\b+1}
   D^{\a\b}_{\phi(t)}
   dt}+\int_{t_{\eta }}^{t_{\w}}{t^{\a+\b+2}\cos^{\a+2}(\phi(t))\sin^{\b}(\phi(t))\pphi(t)dt},
  \]
and, integrating by parts in the first integral,  this identity reads
  \begin{align*}
   I_{\w\eta}^{\a\b}=&\left[\frac{t^{\a+\b+2}
    D^{\a\b}_{\phi(t)}}{\a+\b+2}\right]^{t_{\w}}_{t_{\eta}}
   +\frac{\a+1}{\a+\b+2}\int_{t_{\eta}}^{t_{\w}}{t^{\a+\b+2}\cos^\a(\phi(t))\sin^{\b+2}(\phi(t))\pphi(t)dt}\\
   &-\frac{\b+1}{\a+\b+2}\int_{t_{\eta}}^{t_{\w}}{t^{\a+\b+2}\cos^{\a+2}(\phi(t))\sin^{\b}(\phi(t))\pphi(t)dt}\\
   &+\int_{t_{\eta}}^{t_{\w}}{t^{\a+\b+2}\cos^{\a+2}(\phi(t))\sin^{\b}(\phi(t))\pphi(t)dt}.
  \end{align*}
  Grouping the trigonometric terms and then performing the change of variable $\phi(t)=\0$, we get
  \begin{align*}
   I_{\w\eta}^{\a\b}=&\left[\frac{t_\0^{\a+\b+2}
    D^{\a\b}_{\0} }{\a+\b+2}\right]^\w_{\eta}
   +\frac{\a+1}{\a+\b+2}\int_{\eta}^{\w}{t_\0^{\a+\b+2}\cos^\a(\0)\sin^\b(\0)d\0}.
  \end{align*}
  This is our claim.
 \end{proof}

 For $\a,\b\in\N$, $h\in\N$ and $\eta \in (0,\pi/4)$ we let
 \begin{equation} \label{juwel}
    j^{\a\b}_{h\eta} = \eta^\beta \int_{2 h \pi} ^{ 2h\pi +\eta } t_\vartheta^{\a+\b+2} \, d\vartheta=\int_{t_{h\eta}}^{t_h}t^{\a+\b+2}|\pphi(t)|dt,
 \end{equation}
 where in the second equality we let $\vartheta=\phi(t)$.

 \begin{corollary} There exist constants $0<c_{\a\b}< C_{\a\b}$ depending on $\a,\b\in\N$ such that for all
  $h\in\N$ and $\eta \in (0,\pi/4)$ we have
  \begin{equation}
  \label{stingaling}
   c_{\a\b}
    j^{\a\b}_{h\eta}  \leq | I^{\a\b}_{2h\pi, 2h\pi+\eta}|\leq C_{\a\b}     j^{\a\b}_{h\eta} .
  \end{equation}
 \end{corollary}

 \begin{proof} From \eqref{I->J} with $D^{\a\b}_{2h\pi } =0 $ we obtain
 \[
    (\a+\b+2) | I_{2h\pi, 2h\pi+\eta  }^{\a\b}|
    =
    t_{2h\pi +\eta}^{\a+\b+2} D^{\a\b}_{\eta}
   +(\a+1)J_{2h\pi, 2h\pi+\eta}^{\a\b},
 \]
 where $c_{\a\b} \eta^{\b+1}\leq  D^{\a\b}_\eta \leq \eta ^{\b+1}$, because $\eta\in (0,\pi/4)$, and
 \[
 c_{\a\b}
 \eta^{\b+1}  t_{2h\pi+\eta }^{\a+\b+2}
 \leq
  c_{\a\b} \eta ^\beta
  \int_{2 h \pi} ^{ 2h\pi +\eta }
  t_\0^{\a+\b+2}
  d\0\leq
  J_{2h\pi, 2h\pi+\eta}^{\a\b}
  \leq    \eta^\beta
  \int_{2 h \pi} ^{ 2h\pi +\eta }
  t_\0^{\a+\b+2} d\0 .
 \]
 The claim follows.
 \end{proof}

 \begin{remark}
\label{TECH}  
We will use the estimates \eqref{stingaling} in the proof of the solvability of the end-point equations. In particular, the computations above are  possible thanks to the structure of the monomials $p_i$: here, their  dependence only on the variables $x_1$ and $x_2$, ensured by 
\eqref{12ML}, is crucial.
 When the coefficients $a_i$ depend on all the variables $x_1,\ldots,x_n$, repeating the same computations seems difficult. Indeed, 
 in the integrals \eqref{I} and \eqref{J} there are also the coordinates $\gamma_3,\ldots,\gamma_n$. Then, the new identity \eqref{I->J} becomes more complicated because other addends appear after the integration by parts, owing to the derivatives of $\gamma_3,\ldots,\gamma_n$. Now, by the presence of these new terms  the estimates from below in \eqref{stingaling} are difficult, while the estimates from above remain possible.
\end{remark}

We denote by   $\k_\e$ the rigid translation by $\e\in\R $ in the $x_1$ direction of the curve $\k$. Namely, we let  $\k_{\e,1}=\k_1+\e$ and  $\k_{\e,2}=\k_2$.
Recall the notation $t_h =\psi (2\pi h)$ and $t_{h\eta} =\psi(2\pi h+\eta)$,
for $h\in\N$ and $\eta>0$.
In particular, when we take $\e_j$, $h_j$ and $\eta_j$ related  to the $j$-th correction-device, we have $\k_{\e_j}|_{B_j}=\bar\k|_{B_j}$.

In the study of the polynomial part of integrals in \eqref{erri4}  we
need estimates for the quantities 
\[
\Delta_{h\eta\e}^{\a\b}
 =\int_{\k_\e|_{[t_{h\eta},t_h]}}{p_{\a\b}(x_1,x_2)dx_2}-
 \int_{\k|_{[t_{h\eta},t_h]}}{p_{\a\b}(x_1,x_2)dx_2}.
\]

\begin{lemma}
We have
 \begin{equation}
  \begin{split}
  \label{DELTA}
\Delta_{h\eta\e}^{\a\b} =
  (\alpha+1) \e I_{2h\pi,2h\pi+\eta}^{\a-1,\b}+ O  (\e^2),
 \end{split}
 \end{equation}
 where $O(\e^2)/\e^2$ is bounded as $\e\to 0$.
 \end{lemma}

\begin{proof}
The proof is an elementary computation:
\begin{equation}
  \begin{split}
  \label{DELTA1}
\Delta_{h\eta\e}^{\a\b}
 & =\int_{t_{h\eta}}^{t_h}
  \dot \k_2(t) \k_2(t)^{\b}
  \big[(\k_1(t)+\e)^{\a+1}-
\k_1(t)^{\a+1}\big]  dt
\\
  &=\sum_{i=0}^\a\binom{\a+1}{i}\e^{\a+1-i}\int_{t_{h \eta}}^{t_h}{\dot \k_2(t)\k_1(t)^i\k_2(t)^{\b}dt}\\
  &=\sum_{i=0}^\a\binom{\a+1}{i}\e^{\a+1-i}I_{2h\pi,2h\pi+\eta}^{i-1,\b}
  \\
  &
  =
  (\alpha+1) \e I_{2h\pi,2h\pi+\eta}^{\a-1,\b}+ O (\e^2).
 \end{split}
 \end{equation}
\end{proof}.

We estimate the terms in \eqref{erri2}. The quantities $\Delta_i^\gamma$ are introduced in \eqref{DELTA}.

\begin{lemma}
  \label{stimacut}
  Let $\gamma$ be a horizontal spiral with phase $\phi$. For all $i=3,\ldots, n$ and for all $k\in\N$ large enough we have
  \begin{equation}
  \label{eqstimacut}
\Big |   \int_{F_k}  \Delta_i^\g dt \Big| \leq
 \int_{F_k}   t ^{\a_i+\b_i +2} |\dot\phi |  dt  .
  \end{equation}

 \end{lemma}

\begin{proof}  By \eqref{I->J} with vanishing boundary contributions, we obtain
\[
  \begin{split}
    \Big|  \int_{F_k}  p_i(\kappa)  \dot\kappa_2 dt \Big|
    &
     = |   I_{2k \pi, 2(k+1)\pi } ^{\a_i\b_i} | =
     \frac{\a_i+1}{\a_i+\b_i+2}
     |   J_{2k \pi, 2(k+1)\pi } ^{\a_i\b_i} |
     \\
     &
     \leq
    \frac{\a_i+1}{\a_i+\b_i+2}
      \int_{F_k}   t ^{\a_i+\b_i +2} |\dot\phi |  dt,
  \end{split}
\]
so we are left with the estimate of the integral of $r_i$. Using $\kappa_2 = t \sin(\phi(t))$ we get
\[
  \begin{split}
     \int_{F_k} r_i (\kappa) \dot \kappa_2 dt &
     =
     \int_{F_k} r_i (\kappa) (\sin (\phi) + t \cos(\phi )\dot\phi ) dt
     \\
     &
     =
     \int_{F_k}  (t r_i (\kappa) -R_i )  \cos(\phi )\dot\phi dt ,
     \end{split}
\]
where we let
\[
    R_i(t) = \int_{t_{k+1}}^{t} r_i(\kappa) ds.
\]

From \eqref{eq2},  we have $|\kappa (t)| \leq   t$ for all $t\in [0,1]$. 
By part (ii) of Theorem \ref{giallo} we have $|r_i(x)|\leq  C \| x\| ^{w_i}$ for all $x\in\R^n$ near $0$, with $w_i = \a_i+\b_i+2$. It follows that $|r_i(\kappa (t))|\leq C t ^{w_i}$ for all $t\in [0,1]$,  and $|R_i(t)| \leq C t^{w_i+1}$.
We deduce that
\[
 \Big|
 \int_{F_k} r_i (\kappa) \dot \kappa_2 dt \Big|\leq  C  \int_{F_k}  t^{\a_i+\b_i+3 }  |\dot \phi| dt,
\]
and the claim follows.
\end{proof}

Now we study the integrals in \eqref{erri4}.  Let us introduce the following   notation
\[
 \De_{r_i}^\g=\big(r_i(\k)-r_i(\bar\k)\big)\kk_2\qquad \textrm{and}\qquad  
  \de_{r_i}^\g=r_i(\k)-r_i(\bar\k).
 \]

\begin{lemma}
 \label{stimacor}
 Let $\g$ be a horizontal spiral with phase $\phi$. Then for any $j=3,\.,n$ and for $|\e_j|<t_{h_j\eta_j}$,  we have
 \begin{equation}
 \Big|\int_{B_j}\De_{r_i}^\g(t)dt\Big|\leq C  |\e_j|\int_{B_j} t^{w_i}|\pphi(t)|dt,
 \end{equation}
 where $C>0$ is constant. 
\end{lemma}

\begin{proof} For $t\in B_j$ we have $\k_2(t)=\bar\k_2(t)$ and $\bar\k_1(t)=\k_1(t)+\e_j$. By Lagrange Theorem it follows that
 \[
  \de_{r_i}^\g(t)=\e_j\d_1r_i(\k^*(t)),
 \]
 where $\k^*(t)=(\k_1^*(t),\k_2(t))$ and $\k_1^*(t)=\k_1(t)+\de_j$, $0<\de_j<\e_j$. By Theorem \ref{giallo} we have  $|\d_1r_i(x)|\leq C \|x\|^{w_i-1}$ and so, also using $\delta_j<\e_j<t$,
 \[
  |\d_1r_i(\k^*(t))|\leq C \|\k^*(t)\|^{w_i-1}= C \Big(|\k_1(t)+\de_j|+|\k_2(t)|\Big)^{w_i-1}\leq  C t^{w_i-1}.
 \]
 This implies $|\de_{r_i}^\g(t)|\leq  C  |\e_j|t^{w_i-1}$.

 Now, the integral we have to study is
 \begin{align*}
  \int_{B_j}\De_{r_i}^\g dt = \int_{B_j}\de_{r_i}^\g\kk_2dt = \int_{B_j}\de_{r_i}^\g\sin\phi dt+\int_{B_j}\de_{r_i}^\g t\pphi\cos\phi dt.
 \end{align*}
 We integrate by parts the integral  without $\pphi$, getting
 \[
  \int_{B_j}\de_{r_i}^\g\sin\phi dt=\Big[\sin\phi(t)\int_{t_{h_j\eta_j}}^t\de_{r_i}^\g ds\Big]_{t=t_{h_j\eta_j}}^{t=t_{h_j}}-
  \int_{B_j}\Big\{\pphi\cos\phi\int_{t_{h_j\eta_j}}^t\de_{r_i}^\g ds\Big\} dt.
 \]
 Since the boundary term is 0, we obtain
 \[
   \int_{B_j}\de_{r_i}^\g\kk_2dt=\int_{B_j}\Big\{t\de_{r_i}^\g-\int_{t_{h_j\eta_j}}^t\de_{r_i}^\g ds\Big\}\pphi\cos\phi dt,
 \]
 and thus 
 \begin{align*}
  \Big|\int_{B_j}\de_{r_i}^\g\kk_2dt\Big|&\leq
  \int_{B_j}\Big\{t|\de_{r_i}^\g|+\int_{t_{h_j\eta_j}}^t|\de_{r_i}^\g| ds\Big\}|\pphi|dt 
  \leq C |\e_j|\int_{B_j}t^{w_i}|\pphi|dt.
 \end{align*}
\end{proof}

\begin{remark}
\label{TECH2}  
We stress again the fact that, when the coefficients $a_i$ depend on all the variables $x_1,\ldots,x_n$,  the computations above become less clear. As a matter of fact, there is a  non-commutative effect of the devices due to the varying coordinates $\gamma_3,\ldots,\gamma_n$ that modifies the coefficients of the parameters $\e_j$.
\end{remark}

\section{Solution to the end-point equations}

\label{SYS}

In this section we solve the system of equations $E_i^{k,\bar{\mathcal E}}=0$,    $i=3,\ldots,n$. The homogeneous polynomials $p_j$ are of the form $p_j(x_1,x_2) = x_1^{\a_j+1} x_2 ^ {\b_j}$, as in \eqref{pipi}.

The quantities \eqref{erri2}, \eqref{erri3} and \eqref{erri4} are, respectively,
 \begin{equation}
\begin{split}
&\int_{F_k} \Delta_i^\g dt = I^{\a_i\b_i}_k+\int_{F_k} r_i(\k(t)) dt,
\\
 &\int_{A_j} \Delta_i ^\g dt = 0,
 \\
 &\int_{B_j} \Delta_i ^\g dt
 = -\Delta ^{\a_i\b_i}_{{h_j}\eta_j\e_j} + \int_{B_j}\De_{r_i}^\g dt,
 \end{split}
 \end{equation}
 where we used the short-notation $I^{\a_i\b_i}_k=I^{\a_i\b_i}_{2\pi k, 2\pi (k+1) }$.
So  the end-point equations $E_i^{k,\bar{\mathcal E}}=0$ read
 \begin{equation} \label{EQFIN}
 f_i(\e) = b_i ,
 \quad i=3,\ldots,n.
 \end{equation}
 with
 \[
f_i(\e) =  
  \sum _{j=3}^n \Big(\Delta ^{\a_i\b_i}_{{h_j}\eta_j\e_j}-\int_{B_j}\De_{r_i}^\g dt\Big)
  \quad \textrm{and}
  \quad
  b_i=
  \int_{F_k} \Delta_i^\g dt.
 \]
  We will regard $k$, ${h_j}$, and $\eta_j$
  as parameters and we will solve the system of equations \eqref{EQFIN} in the unknowns $\e=(\e_3,\ldots,\e_n)$. The functions $f_i:\R^{n-2}\to\R$ are analytic and the data $b_i$ are estimated from above by  \eqref{eqstimacut}:
  \begin{equation}
  \label{b_i}
  |b_i| \leq  
 \int_{F_k}   t ^{w_i } |\dot\phi |  dt  .
   \end{equation}

  \begin{theorem} \label{stilx}
  There exist real parameters   $\eta_3,\ldots, \eta_n>0$ and integers
  $h_3>\ldots>h_n$   such that  for all $k\in \N $ large enough the system of equations \eqref{EQFIN} has a unique solution $\e = (\e_3,\ldots,\e_n)$ satisfying
  \begin{equation}
  \label{biba}
  |\e|\ \leq C \sum_{i=3}^n
| b_i |,
  \end{equation}
  for a constant $C>0$ independent of $k$.
  \end{theorem}

  \begin{proof}
  We will use the inverse function theorem. Let $A=\big( a_{ij}\big) _{i,j =3,\ldots,n}\in M_{n-2} (\R)$ be  the Jacobian matrix of $f= (f_3,\ldots, f_n)$ in the variables $\e=(\e_3,\ldots,\e_n)$ computed at 
  $\e=0$. By \eqref{DELTA} and Lemma \ref{stimacor}
  we have
  \begin{equation} \label{aij}
  a_{ij} = \frac{\partial f_i(0) }{\partial \e_j} = 
 (\a_i+1)
  I_{h_j\eta_j}^{\a_i-1,\b_i}  + o(I_{h_j\eta_j}^{\a_i-1,\b_i} ).
  \end{equation}
  Here, we are using the fact that for $h_j\to\infty$ we have
  \[
   \int_{B_j}   t ^{w_i } |\dot\phi |  dt = o \Big( 
   \int_{B_j}   t ^{w_i -1} |\dot\phi |  dt\Big).
   \]
   The proof of Theorem \ref{stilx}
 will be complete if we show that the matrix $A$ is invertible.

We claim that there exist real parameters $\eta_3,\ldots, \eta_n>0$ and positive integers
  $h_3>\ldots>h_n$ such that
  \begin{equation}
\label{detto}
 \det (A) \neq 0.
    \end{equation}

 The proof is by induction on  $n$. When   $n=3$,  the matrix $A$ boils
  down to the  real number $a_{33}$. From \eqref{aij} and  \eqref{stingaling} we deduce that for any $\eta_3\in (0,\pi/4)$ we have
  \begin{equation}
  \begin{split}
   \label{stile}
 |a_{3 3}|  &
 \geq \frac 12  (\a_3+1) | I_{h_3\eta_3}^{\a_3-1,\b_3}|\geq  c_{\a\b}
 j^{\a_3-1,\b_3} _{ h_3\eta_3}>0.
\end{split}
  \end{equation}
   We can choose $h_3\in \N $ as large as we wish.

  Now we prove the inductive step. We assume that   \eqref{detto}
 holds when $A$ is a $(n-3)\times (n-3)$ matrix, $n\geq 4$. We develop $\det(A)$ with respect to the first column using  Laplace formula:
 \begin{equation}
  \label{eq33}
  \begin{split}
   \det( A) = \sum_{i=3}^n(-1)^{i+1} a_{i3}  P_i ,
  \end{split}
 \end{equation}
 where
 \[
  P_i=P_i( a_{43}  ,\dots, a_{4n} ,\dots,
   \hat{a}_{ i3} ,\dots,\hat{a}_{in} ,\dots,
   a_{n3}   ,\dots,a_ {nn})
   \]
    are the   determinants of the minors.
 By the inductive assumption, there exist $\eta_4,\ldots,\eta_{n}\in (0,\pi/4)$ and  integers $ h_4>\dots>h_n $ such that $|P_i|>0$.
 By \eqref{stingaling}, for any  $\eta_3\in (0,\pi/4)  $ we have the estimates
 \begin{equation}
 \label{pinco}
 c_0  j ^{\a_i-1,\b_i}_{h_3\eta_3} \leq
  |a_{i3}|
 \leq
 C_0  j ^{\a_i-1,\b_i}_{h_3\eta_3} ,
 \end{equation}
for absolute constants $0<c_0<C_0$.
 The leading (larger) $|a_{i3}|$ can be found in the following way.
 On the set   $\mathcal A = \{ (\a_i,\b_i) \in \N\times \N : i=3 ,\ldots, n\}$ we introduce the order  $(\a,\b)<(\a',\b')$ defined by  the conditions
 $\a+\b < \a'+\b'$, or $\a+\b = \a'+\b'$ and  $\b<\b'$.
  We denote by $(\a_\iota,\b_\iota)\in \mathcal A$, for some $\iota =3,\ldots,n$, the minimal element with respect to this order relation.

 We claim that,  given $\e_0>0$, for all $h_3 > h_4$ large enough and for some $ 0< \eta_3<\pi/4$  the following inequalities  hold:
\begin{equation}
\label{stillox}
|a_{i3} | |P_i|\leq  \e_0|a _{\iota 3}P_{\iota}|, \quad \textrm{for}\quad i\neq \iota .
\end{equation}
In the case when $i=3,\ldots,n$ is such that $\a_i+ \b_i =\a_\iota+\b_\iota $, then we have $\b_i>\b_\iota$.
By \eqref{pinco} and \eqref{juwel}, inequality \eqref{stillox} is implied by $\eta_3^{\b_i -\beta_\iota} |P_i| \leq \e_0|P_\iota|$, possibly for a smaller $\e_0$. So we fix $\eta_3\in (0,\pi/4)$ independently from $h_3$ such that
\[
 0<\eta_3 \leq \min \Big\{ \Big(\frac{\e_0
   |P_\iota|}{ |P_i|}\Big) ^{ 1/(\b_i-\b_\iota)}: i\neq \iota \Big\}.
\]

In the case when $i=3,\ldots,n$ is such that $\a_i+ \b_i >\a_\iota+\b_\iota  $,
inequality \eqref{stillox} is implied by
\[
\int_{B_3}t^{\a_i+\b_i }|\pphi(t)| dt \leq
   \e_0\eta_3 ^{\b_\iota-\b_i }\frac{|P_\iota|}{|P_i|}
\int_{B_3} t^{\a_\iota+\b_\iota}|\pphi(t)| dt.
\]
This  holds for all  $h_3\in \N$ large enough. 

Now we can estimate from below the determinant of $A$ using \eqref{stillox}. We have
 \[
 |\det(A)| \geq | a _{\iota 3}  P_{\iota}|- \sum_{i\neq \iota} |a_{i3} | |P_i|\geq
 \frac 12 | a _{\iota3}  P_{\iota}|
 \]
 and the last inequality holds for all $h_3\in \N$ large enough, after fixing    $\eta_3>0$.
 This ends the proof of the theorem.
 \end{proof}

  \section{Nonminimality of the spiral}
  
  \label{seven}
 
 In this section we prove Theorem \ref{MAIN}. Let $\gamma\in AC([0,1];M)$ be a horizontal   spiral of the form \eqref{eq2}. We work in exponential coordinates of the second type centered at $\gamma(0)$.

We fix on $\D$ the metric $g$ making orthonormal the vector fields $X_1$ and $X_2$ spanning $\D$.
This is without loss of generality, because any other metric is equivalent to this one in a neighborhood of the center of the spiral.
With this choice, the length of  $\gamma$ is the standard length of its horizontal coordinates and for a spiral as in \eqref{eq2} we have
\begin{equation}
  \label{eq3}
    L(\gamma) =\int_0^1{|\dot{\kappa}(t)|dt} 
  =\int_0^1{\sqrt{1+t^2\pphi(t)^2}}dt.
 \end{equation}
In particular,   $\gamma$ is rectifiable precisely when   $t\dot\phi \in L^1(0,1)$,
and $\k$ is a Lipschitz curve in the plane precisely when  $t \dot\phi \in L^ \infty(0,1)$.

  For  $k\in \N$ and $\bar {\E} = (\E_3,\ldots,\E_n)$,
 we denote by $\Dev(\gamma;k,\bar{\E})$ the curve constructed in  Section \ref{sez2}.
 The devices  $\E_j =(h_j,\eta_j,\e_j)$
 are chosen in such a way that the parameters $h_j,\eta_j$ are fixed as in Theorem \ref{stilx}
 and $\e_3, \ldots,\e_n$ are the unique solutions to the system \eqref{EQFIN}, for $k$ large enough. In this way the curves $\gamma$ and $  \Dev(\gamma;k,\bar{\E})(1)$ have the same initial and end-point.

 We claim that for $k\in\N$ large enough the length of $\Dev(\gamma;k,\bar{\E})$ is less than the length of $\gamma$. We denote by $
 \Delta L(k)  =
    L(\Dev(\gamma;k,\bar{\E})) -  L(\gamma) $
    the    gain of length and, namely, 
 \begin{equation} \label{vieri}
 \begin{split}
 \Delta L(k)  
 &
 =
 \int_{F_k}
  \sqrt{1+t^2\pphi(t)^2}
  dt
 -\Big( t_k-t_{k+1} +2 \sum_{j=3}  ^n
 |\e_j| \Big)
\\
& =
 \int_{F_k}
 \frac{t^2\pphi(t)^2}
 {\sqrt{1+t^2\pphi(t)^2}+1}
 dt-2 \sum_{j=3}  ^n
 |\e_j|.
 \end{split}
 \end{equation}
By \eqref{biba}, there exists a constant $C_1>0$ independent of $k$ such that 
the solution
$\e=(\e_3, \ldots,\e_n)$ to the  end-point equations   \eqref{EQFIN}
satisfies 
\begin{equation}
\label{viola}
|\e| \leq C_1 \sum_{i=3}^n | I_{ k } ^{\a_i\b_i} | \leq 
C_2 \sum _{i=3}^n  \int_{F_k } t^{w_i}|\dot\phi(t)| dt
\leq C_3  \int_{F_k} t^{2}|\dot\phi(t)| dt.
\end{equation}
We used \eqref{stingaling} and the fact that $w_i\geq 2$.
The new constants $C_2,C_3$ do not depend on $k$.

By \eqref{vieri} and \eqref{viola}, the inequality $\Delta L(k)> 0$ is implied by
\begin{equation}
\label{DIS1}
 \int_{F_k}
 \frac{t^2\pphi(t)^2}
 {\sqrt{1+t^2\pphi(t)^2}+1}
 dt>C_4  \int_{F_k} t^{2}|\dot\phi(t)| dt,
\end{equation}
 where $C_4$ is a large constant independent of $k$. 
 For any $k\in \N$, we split the interval   $  F_k= F_k^+\cup F_k^-$  where 
  \[
   F_k ^+ =  \{t\in F_k:  
    | t \pphi(t)|\geq 1\}
   \quad
   \textrm{and}
   \quad
  F _k ^- =   \{t\in F_k: | t \pphi(t)|<1\}.
 \]
 On the set   $F_k^+$ we have  
 \begin{equation} \label{INES1}
 \begin{split}
  \int_{F_k^+  }
  \frac{t^2\pphi(t)^2}
 {\sqrt{1+t^2\pphi(t)^2}+1} dt \geq 
\frac 13    \int_{F_k^+  }
 {t|\pphi(t)|}
  dt
  \geq  C_4  \int_{F_k^+  }
 {t^2 |\pphi(t)|}
  dt,
  \end{split}
  \end{equation}
  where the last inequality holds for all $k\in \N$ large enough, and namely as soon as 
    $ 3 C_4  t_k <1$.
    On the set   $F_k^-$  we have  
 \begin{equation}
\label{INES2}
 \begin{split}
  \int_{F_k^-  }
  \frac{t^2\pphi(t)^2}
 {\sqrt{1+t^2\pphi(t)^2}+1} dt \geq 
\frac 13   \int_{F_k^+  }
 {t^2 |\pphi(t)|^2}
  dt
  \geq    
C_4    \int_{F_k^-  }
 {t^2 |\pphi(t)|}
  dt,
  \end{split}
  \end{equation}
  where the last inequality holds for all   $k\in \N$ large enough, by our assumption on the spiral
  \[
  \lim_{t\to0^+}  |\dot\phi(t) |=\infty.
  \]
  
  Now \eqref{INES1} and \eqref{INES2} imply \eqref{DIS1} and thus $\Delta L(k)>0$.
  This ends the proof of Theorem \ref{MAIN}.

\end{document}